\documentclass[12pt]{amsart}
\usepackage{amsmath}
\usepackage{amsfonts}
\usepackage{amssymb}
\usepackage{amscd}
\usepackage{mathtools}
\usepackage{amsxtra}
\usepackage{amsthm}
\usepackage{graphicx}
\usepackage[abbrev,alphabetic]{amsrefs}
\PassOptionsToPackage{dvipsnames}{xcolor}
\usepackage{soul,xcolor}
\setstcolor{red}
\usepackage{stmaryrd}
\usepackage{booktabs}
\usepackage{multirow}
\newtagform{tiny}{\tiny(}{)}

\usepackage{mathtools}
\usepackage{hyperref}
\usepackage[margin=1.25in]{geometry}
\usepackage{mathrsfs}

\usepackage{amsthm}
\usepackage{comment}
\usepackage[all,cmtip]{xy}
\usepackage{tikz-cd}
\usetikzlibrary{cd}

\tikzcdset{
  cells={font=\everymath\expandafter{\the\everymath\displaystyle}},
}

\usepackage[all]{xy}

\makeatletter
\def\@tocline#1#2#3#4#5#6#7{\relax
  \ifnum #1>\c@tocdepth % then omit
  \else
    \par \addpenalty\@secpenalty\addvspace{#2}%
    \begingroup \hyphenpenalty\@M
    \@ifempty{#4}{%
      \@tempdima\csname r@tocindent\number#1\endcsname\relax
    }{%
      \@tempdima#4\relax
    }%
    \parindent\z@ \leftskip#3\relax \advance\leftskip\@tempdima\relax
    \rightskip\@pnumwidth plus4em \parfillskip-\@pnumwidth
    #5\leavevmode\hskip-\@tempdima
      \ifcase #1
       \or\or \hskip 1em \or \hskip 2em \else \hskip 3em \fi%
      #6\nobreak\relax
    \hfill\hbox to\@pnumwidth{\@tocpagenum{#7}}\par% <---- \dotfill -> \hfill
    \nobreak
    \endgroup
  \fi}
\makeatother

\renewcommand{\mod}{\ \textrm{mod}\ }

\newcommand{\Z}{\mathbb{Z}}
\newcommand{\Q}{\mathbb{Q}}

\newcommand{\cHom}{\mathcal{H}om}
\newcommand{\bQ}{\mathbb{Q}}

\newcommand{\cA}{\mathcal{A}}

\newcommand{\cK}{\mathcal{K}}

\newcommand{\cO}{\mathcal{O}}

\newcommand{\sO}{\mathcal{O}}

\newcommand{\m}{\mathfrak{m}}
\newcommand{\n}{\mathfrak{n}}

\newcommand{\Frac}{\mathrm{Frac}}
\newcommand{\Proj}{\mathrm{Proj}}
\newcommand{\wt}{\widetilde}

\def\var{\overline}

\DeclareMathOperator{\Spec}{Spec}

\DeclareMathOperator{\Hom}{Hom}

\DeclareMathOperator{\Ker}{Ker}

\DeclareMathOperator{\sht}{ht}

\newtheorem{theoremA}{Theorem}

\theoremstyle{remark}
\newtheorem*{ackn}{Acknowledgements}

\theoremstyle{plain}

\numberwithin{equation}{section}

  \newsavebox{\pullbackdl}
\sbox\pullbackdl{%
\begin{tikzpicture}%
\draw (-1ex,0ex) -- (0ex,0ex);%
\draw (0ex,-1ex) -- (0ex,0ex);%
\end{tikzpicture}}

\newsavebox{\pushoutdr}
\sbox\pushoutdr{%
\begin{tikzpicture}%
\draw (-1ex,-1ex) -- (-1ex,0ex);%
\draw (-1ex,0ex) -- (0ex,0ex);%
\end{tikzpicture}}

\usepackage{thmtools}

\usepackage[capitalise,noabbrev]{cleveref}

\declaretheorem[name=Theorem,numberwithin=section]{theorem}
\declaretheorem[sibling=theorem,name=Lemma]{lemma}
\declaretheorem[sibling=theorem,name=Proposition]{proposition}
\declaretheorem[sibling=theorem,name=Corollary]{corollary}

\declaretheorem[sibling=theorem,style=definition,name=Definition]{definition}
\declaretheorem[sibling=theorem,style=definition,name=Example]{example}
\declaretheorem[sibling=theorem,style=definition,name=Notation]{notation}
\declaretheorem[sibling=theorem,style=remark,name=Remark]{remark}



\title[Quasi-$F$-splitting and smooth weak del Pezzo]{Quasi-$F$-splitting and smooth weak del Pezzo surfaces in mixed characteristic}

\author{Hirotaka Onuki}
\address{The University of Tokyo, Tokyo 153-8914, Japan}
\email{onuki-hirotaka010@g.ecc.u-tokyo.ac.jp}

\author{Teppei Takamatsu}
\address{Department of Mathematics (Hakubi Center),
Kyoto University,
Kitashirakawa, Oiwake-cho, Sakyo-ku,
Kyoto 606-8502, JAPAN}
\email{teppeitakamatsu.math@gmail.com}

\author{Shou Yoshikawa}
\address{Institute of Science Tokyo, Tokyo 152-8551, Japan}
\email{yoshikawa.s.9fe9@m.isct.ac.jp}
\begin{document}

\begin{abstract}
We introduce the notion of quasi-$F$-splitting in mixed characteristic and study Kodaira-type vanishing on quasi-$F$-splitting varieties.
As an application, we prove a Kodaira-type vanishing on lifts of rational double point (RDP) del Pezzo surfaces. 
\end{abstract}

%We prove that every such lift is quasi-$F$-split and, as a consequence, obtain the quasi-$F$-splitting of smooth weak del Pezzo surfaces. 
%Furthermore, we also establish a Kodaira-type vanishing theorem for liftable $\Q$-Cartier ample Weil divisors on RDP del Pezzo surfaces. 

\maketitle

\section{Introduction}

In recent years, algebraic geometry in mixed characteristic has been actively developed (cf.~\cite{BMPSTWW23}, \cite{TY23}, \cite{HLS}).
Many of the results obtained so far can be regarded as mixed characteristic analogues of those in positive characteristic.
In this paper, we observe that, in certain situations, mixed characteristic geometry actually exhibits better behavior than its positive characteristic counterpart.
To explain this phenomenon, we study a mixed characteristic analogue of the notion of quasi-$F$-splitting.

Quasi-$F$-splitting in mixed characteristic was recently introduced by \cite{Yoshikawa25} as an analogue of the notion of quasi-$F$-splitting in positive characteristic due to \cite{Yobuko19}.
In this paper, we formulate a global analogue of this concept, providing a new framework to study birational and cohomological properties of varieties in mixed characteristic via Frobenius-theoretic techniques.
As an application, we prove a Kodaira-type vanishing theorem for quasi-$F$-split varieties in mixed characteristic as follows:

\begin{theoremA}[cf.~\cref{thm:KV-vanishing}]\label{intro:thm:KV-vanishing}
Let $k$ be an algebraically closed field of characteristic $p>0$.
Let $\var{X}$ be a normal projective variety over $k$ and $X$ a lift of $\var{X}$ over $W(k)$.
Assume that $X$ is quasi-$F$-split and $X$ is Cohen-Macaulay.
Let $L$ be a $\Q$-Cartier Weil divisor such that $A:=L-K_X$ is ample. 
Suppose that $\cO_X(-p^eA)$ is Cohen–Macaulay for every $e \geq 0$.
Then
\[
H^j(X,\cO_X(L)) = 0
\quad \text{for all } j>0.
\]
\end{theoremA}

Building on this framework, we further examine the quasi-$F$-splitting of RDP (rational double point) del Pezzo surfaces.
In particular, we observe that certain pathologies in positive characteristic disappear after lifting to mixed characteristic.

It is known that, in positive characteristic, RDP  del Pezzo surfaces are not quasi-$F$-split in general, and accordingly Kodaira-type vanishing fails for such surfaces (cf.~\cite{CT18}*{Proposition~4.3~(iii)}, \cite{KN22}*{Theorem~1.7~(3)}).
One of the aims of this paper is to show that, after lifting to mixed characteristic, the situation improves: the corresponding surfaces become quasi-$F$-split and suitable vanishing theorems hold.

First, we prove the quasi-$F$-splitting of lifts of RDP del Pezzo surfaces to mixed characteristic.

% Quasi-$F$-splitting in mixed characteristic was recently introduced in \cite{Yoshikawa25} as a mixed characteristic analogue of the classical notion of quasi-$F$-splitting in positive characteristic. 
% A global analogue of quasi-$F$-splitting in mixed characteristic can be defined in a manner similar to \cite{Yobuko19} in the positive characteristic setting. 
% This provides a new framework for studying birational and cohomological properties of varieties in mixed characteristic by means of Frobenius-theoretic techniques. 

% It is well known that in positive characteristic, RDP (rational double point) del Pezzo surfaces are \emph{not} quasi-$F$-split in general, and accordingly Kodaira-type vanishing fails for such surfaces (\cite{KN22}*{Theorem~1.7~(3)}, \cite{KTTWYY1}*{Theorem~6.3}). 
% One of the aims of this paper is to show that, after lifting to mixed characteristic, the situation improves: the corresponding surfaces become quasi-$F$-split and suitable vanishing theorems hold. 

% First, we investigate the quasi-$F$-splitting of lifts of RDP del Pezzo surfaces to mixed characteristic. 

\begin{theoremA}[\cref{RDP-dP}]\label{intro:RDP-dP}
Let $\var{X}$ be an RDP del Pezzo surface over an algebraically closed field $k$ of characteristic $p>0$.
Let $X$ be a lift of $\var{X}$ over $W(k)$.
Then $X$ is quasi-$F$-split.
\end{theoremA}

\noindent
As a corollary of Theorem\ref{intro:RDP-dP}, we deduce the quasi-$F$-splitting of smooth weak del Pezzo surfaces over mixed characteristic (\cref{weakdP}).

Combining Theorems \ref{intro:RDP-dP} and \ref{intro:thm:KV-vanishing}, we obtain the following vanishing result:

\begin{theoremA}[\cref{van-CM-ample}]\label{intro:van-CM-ample}
Let $k$ be an algebraically closed field, and let $X$ be a lift of an RDP del Pezzo surface over $W(k)$.
Let $A$ be a nef and big Weil $\Q$-Cartier divisor on $X$ such that $\cO_X(A)$ is Cohen--Macaulay.    
Then we have $H^i(\cO_X(K_X+A))=0$ for $i>0$.
\end{theoremA}

Furthermore, we establish a Kodaira-type vanishing theorem for liftable $\Q$-Cartier ample Weil divisors on RDP del Pezzo surfaces in positive characteristic:

\begin{theoremA}[\cref{vanishing}, cf.~\cite{KN22}*{Theorem~1.7~(3)}]\label{intro:vanishing}
Let $\var{X}$ be an RDP del Pezzo surface over an algebraically closed field $k$ of characteristic $p>0$.
Let $\var{A}$ be a nef and big Weil $\Q$-Cartier divisor on $\var{X}$.
Assume that there exists a lift $X$ of $\var{X}$ over $W(k)$ and a lift $A$ of $\var{A}$ on $X$ such that $A$ is a $\Q$-Cartier divisor.
Then $H^1(\cO_{\var{X}}(-\var{A}))=0$.
\end{theoremA}

\begin{remark}
In \cite{KN22}, the authors study pathologies of RDP del Pezzo surfaces in positive characteristic, from which the following facts related to the main theorem can be deduced.
\begin{enumerate}
    \item Theorem~\ref{intro:RDP-dP} follows immediately from \cite{Yoshikawa25}*{Proposition~4.6}, \cite{KN22}*{Theorem~1.7~(1)}, and \cite{KTTWYY1}*{Theorem~6.3}, except when 
    \[
    (p,\mathrm{Dyn}(X))=(3,4A_2),\ (2,4A_1+D_4),\ (2,8A_1),\ \text{or}\ (2,7A_1).
    \]
In these cases, the classification of the equations was given in \cite{KN23}, which plays a key role in the proof of Theorem~\ref{intro:RDP-dP} together with the Fedder-type criterion (\cite{Yoshikawa25}*{Theorem~4.13}).

    \item Kodaira-type vanishing for ample Weil $\Q$-Cartier divisors holds on RDP del Pezzo surfaces, except when 
    \[
    (p,\mathrm{Dyn}(X))=(2,8A_1)\ \text{or}\ (2,7A_1)
    \]
    by \cite{KN22}*{Theorem~1.7~(3)}.
    Therefore, the above two cases are the only nontrivial cases in Theorem~\ref{intro:vanishing}.
\end{enumerate}

\end{remark}

% \noindent
% These results suggest that mixed characteristic provides a more flexible framework for recovering vanishing theorems and Frobenius-theoretic properties that fail in positive characteristic, and we expect further applications of mixed characteristic quasi-$F$-splitting to the study of higher-dimensional varieties. 

\begin{ackn}
The authors wish to express their gratitude to Ryotaro Iwane, Tatsuro Kawakami, Masaru Nagaoka, and Shunsuke Takagi for their valuable discussion.
The authors also thank Yuya Matsumoto for providing helpful comments.
%He is also grateful to Linquan Ma for helpful comments. 
Takamatsu was Supported by JSPS KAKENHI Grant Number JP25K17228.
Yoshikawa was supported by JSPS KAKENHI Grant number JP24K16889.
\end{ackn}

\section{Preliminaries}
Throughout the paper, unless otherwise specified, 
$p$ denotes a fixed prime number.
In this section, we collect the necessary background material and fix notation used throughout the paper.
We recall the definition of Witt vectors and review their basic algebraic properties,
together with the Frobenius and Verschiebung structures needed for later sections.

\subsection{The ring of Witt vectors}\label{subsec:Witt vectors}
We recall the definition of the ring of Witt vectors. The basic reference is \cite{Serre}*{II,~\S~6}.

\begin{lemma}
\label{lem:witt polynomial}
For $n\in \Z_{\geq 0}$, we define the polynomial $\varphi_{n} \in \Z[X_{0},\ldots, X_{n}]$ by
\[
\varphi_{n} = \varphi_n (X_0, \ldots, X_n) := \sum_{i=0}^{i=n} p^i X_{i}^{p^{n-i}}.
\]
Then, there exist polynomials $S_{n} ,P_{n} \in \Z[X_{0},\ldots, X_{n}, Y_{0}, \ldots, Y_{n}]$ for any $n \in \Z_{\geq 0}$ such that
\[
\varphi_{m} (S_{0}, \ldots, S_{m}) = \varphi_{m}(X_{0}, \ldots X_{m}) +\varphi_{m}(Y_{0}, \ldots Y_{m})
\]
and 
\[
\varphi_{m} (P_{0}, \ldots, P_{m}) = \varphi_{m}(X_{0}, \ldots X_{m}) \cdot \varphi_{m}(Y_{0}, \ldots Y_{m})
\]
hold for any $m \in \Z_{\geq 0}$.
\end{lemma}

\begin{proof}
See \cite{Serre}*{II, Theorem 6}.
\end{proof}

\begin{definition}\label{defn:Witt ring}
Let $A$ be a ring not
necessarily of characteristic $p$.
\begin{itemize}
\item We set
\[
W(A) := \{
(a_{0}, \ldots, a_{n}, \ldots) | a_{n} \in A
\} = \prod_{\Z_{\geq 0}} A,
\]
equipped with the addition
\[
(a_{0}, \ldots, a_{n}, \ldots) + (b_{0}, \ldots, b_{n}, \ldots) := (S_{0}(a_{0},b_{0}), \ldots, S_{n}(a_{0}, \ldots, a_{n}, b_{0}, \ldots, b_{n}), \ldots)
\]
and the multiplication
\[
(a_{0}, \ldots, a_{n}, \ldots) \cdot (b_{0}, \ldots, b_{n}, \ldots) := (P_{0}(a_{0},b_{0}), \ldots, P_{n}(a_{0}, \ldots, a_{n}, b_{0}, \ldots, b_{n}), \ldots).
\]
Then $W(A)$ is a ring with $1= (1, 0, \ldots) \in W(A)$, which is called the \emph{ring of Witt vectors} over $A$.
For $a \in A$, we set $[a]:=(a,0,\ldots ) \in W(A)$.
We define an additive map $V$ by
\[
V \colon W(A) \rightarrow W(A); (a_{0}, a_{1}, \ldots) \mapsto (0,a_{0}, a_{1}, \ldots).
\]
\item We define the ring $W_{n}(A)$ of \emph{Witt vectors over $A$ of length $n$} by
\[
W_{n}(A) := W(A)/ V^{n}W(A).
\]
We note that $V^{n}W(A)\subseteq W(A)$ is an ideal.
For $a \in A$ and $n \in \Z_{\geq 1}$, we set $[a]:=(a,0,\ldots ,0) \in W_n(A)$.
We define a $W(A)$-module homomorphism $R$ by
\[
W(A) \rightarrow W_{n}(A) ; (a_{0}, a_{1}, \ldots) \mapsto (a_{0}, a_{1}, \ldots, a_{n-1})
\]
and call $R$ the restriction map.
Moreover, we define 
\begin{eqnarray*}
V \colon W_{n}(A) \rightarrow W_{n+1}(A),
\end{eqnarray*}
and
\[
R \colon W_{n+1} (A) \rightarrow W_{n}(A),
\]
in a similar way.
The image of $\alpha \in W(A)$ or $W_n(A)$ by $V$ is denoted by $V\alpha$.
\item For any ring homomorphism $f \colon A \to A'$, a ring homomorphism
\[
W(A) \to W(A')\ ;\ (a_0,a_1,\ldots) \mapsto (f(a_0),f(a_1),\ldots)
\]
is induced and the induced homomorphism is commutative with $R$ and $V$.
Note that $f$ also induces a ring homomorphism $W_n(A) \to W_n(A')$.
\item For an ideal $I$ of $A$,
we denote the kernels of $W_n(A) \to W_n(A/I)$ and $W_n(A) \to W_n(A/I)$ induced by the natural quotient map $A \to A/I$ are denoted by $W(I)$ and $W_n(I)$, respectively.
\item We define a ring homomorphism $\varphi$ by
\[
\varphi \colon W(A) \rightarrow \prod_{\Z_{\geq 0}} A; (a_{0}, \ldots, a_{n}, \ldots) \mapsto (\varphi_{0}(a_{0}), \ldots, \varphi_{n}(a_{0}, \ldots, a_{n}), \ldots), 
\]
where $\prod_{\Z_{\geq 0}} A$ is the product in the category of rings.
The map $\varphi$ is injective if $p \in A$ is a non-zero divisor.
The element $\varphi_{n}(a_{0}, \ldots, a_{n})$ (resp.\,the element $\varphi(a_{0}, \ldots a_{n}, \ldots)$) is called the ghost component (resp.\,the vector of ghost components) of $(a_{0}, \ldots a_{n}, \ldots) \in W(A)$.
\item 
%By \cite{Ill79}*{Section~1.3}, there exists a ring homomorphism $F \colon W(A) \to W(A)$ such that 
In \cite{Ill79}*{Section~1.3}, a ring homomorphism $F \colon W(A) \to W(A)$ is defined; it satisfies 
\[
\varphi \circ F (a_0,a_1,\ldots)=(\varphi_1(a_0,a_1),\varphi_2(a_0,a_1,a_2),\ldots)
\]
and moreover it satisfies $f \circ F=F \circ f$ for any ring homomorphism $f \colon A \to A'$.
Furthermore, $F$ induces the ring homomorphism $F \colon W_{n+1}(A) \to W_n(A)$.
\end{itemize}
\end{definition}

\begin{proposition}\label{Frobenius-grade}
Let $A:=\Z[x_0,x_1,\ldots]$ be a graded ring with $\deg (x_i)=p^i$.
We denote
\[
F(x_0,x_1,\ldots) = :(F_0,F_1,\ldots).
\]
Then $F_i$ is homogeneous of degree $p^{i+1}$ for all $i \geq 0$.  
\end{proposition}

\begin{proof}
We prove the claim by induction on $i$.
Note that
\[
\varphi_i \circ F(x_0,x_1,\ldots)
= \varphi_{i+1}(x_0,x_1,\ldots)
= x_0^{p^{i+1}} + p x_1^{p^i} + \cdots + p^{i+1}x_{i+1},
\]
which is homogeneous of degree $p^{i+1}$.
On the other hand, we have
\begin{equation}\label{eq:degree-Frob}
    \varphi_i \circ F(x_0,x_1,\ldots)
    = F_0^{p^i} + p F_1^{p^{i-1}} + \cdots + p^i F_i.
\end{equation}
For $i=0$, we obtain $x_0^p = F_0$, so in particular $F_0$ is homogeneous of degree $p$.
For $i \geq 1$, each $F_j^{p^{i-j}}$ is homogeneous of degree $p^{i+1}$ by the induction hypothesis for $j < i$.
Therefore, by \eqref{eq:degree-Frob}, $F_i$ must also be homogeneous of degree $p^{i+1}$, as desired.
\end{proof}

\begin{proposition}\label{V-F-structure}
Let $A$ be a ring.
Then the map $V \colon F_*W(A) \to W(A)$ is a $W(A)$-module homomorphism.
\end{proposition}

\begin{proof}
First, assume that $A$ is $p$-torsion free.
Then for every $i \geq 1$ and $\alpha, \beta \in W(A)$, we have
\begin{align*}
    \varphi_i \circ V(F(\alpha)\beta)
    &= p\,\varphi_{i-1}(F(\alpha)\beta) \\
    &= p\,\varphi_i(\alpha)\,\varphi_{i-1}(\beta) \\
    &= \varphi_i(\alpha V\beta),
\end{align*}
and moreover $\varphi_0 \circ V = 0$.
Since $\varphi$ is injective, it follows that
\[
V(F(\alpha)\beta) = \alpha V\beta,
\]
as required.

Next, we consider the general case. We may take a surjective ring homomorphism $f \colon B \to A$ such that $B$ is $p$-torsion free.
For lifts $\wt{\alpha},\wt{\beta} \in W(B)$ of $\alpha,\beta \in W(A)$, we have
\[
V(F(\alpha)\beta)
= f\bigl(V(F(\wt{\alpha})\wt{\beta})\bigr)
= f(\wt{\alpha} V\wt{\beta})
= \alpha V\beta
\]
as desired.
\end{proof}

\subsection{Witt divisorial sheaves in mixed characteristic}
Witt divisorial sheaves were defined in \cite{tanaka22}*{Section~3} on normal schemes in positive characteristic.  
In this subsection, we generalize the construction to normal schemes not necessarily of positive characteristic.  
Throughout, $X$ denotes a Noetherian normal scheme.  
Let $\cK_X$ be the sheaf of total quotient rings of $X$.  

Let $X$ be a scheme and $\cA$ a sheaf of rings on $X$.
We define 
%the Witt sheaf
the presheaf
$W_n\cA$ by
\[
W_n\cA(U):=W_n(\cA(U)).
\]
Then $W_n\cA$ is a sheaf of algebras on $X$ and $W_nX:=(X,W_n\sO_X)$ is a scheme.

We now introduce a $W_n\cO_X$-submodule 
\[
W_n\cO_X(D) \subseteq W_n\cK_X
\]
for a $\Q$-divisor $D$.  
We define a subpresheaf $W_n\cO_X(D) \subseteq W_n\cK_X$ by
\[
W_n\cO_X(D)(U) := 
\{(\varphi_0,\ldots,\varphi_{n-1}) \mid 
\varphi_i \in \cO_X(p^iD) \text{ for all } i \}
\subseteq W_n\cK_X(U)
\]
for every open subset $U \subseteq X$.  
%If $S=0$, we write $W_n\cO_X(D)$ instead of $W_n\cI_S(D)$.  

By the same arguments as in \cite{tanaka22}*{Subsection~3.1} and \cite{KTTWYY1}*{Subsection~2.6},  
the following statements hold for $n,m \in \Z_{>0}$:
\begin{enumerate}
\item 
$W_n\cO_X(D)$ is a sheaf.
Moreover, $W_n\cO_X(D)$ is an $S_2$ coherent $W_n\cO_X$-submodule of $W_n\cK_X$ (cf.~\cite{tanaka22}*{Lemma~3.5(1), Proposition~3.8}).
\item 
There exist $W_n\cO_X$-module homomorphisms
\begin{alignat*}{2}
&(\text{Frobenius}) &&F \colon 
W_{n+1}\cO_X(D) \longrightarrow F_*W_{n}\cO_X(pD),  \\ 
&(\text{Verschiebung}) \qquad &&V \colon 
F_*W_n\cO_X(pD) \longrightarrow W_{n+1}\cO_X(D), \\
&(\text{Restriction}) &&R \colon 
W_{n+1}\cO_X(D) \longrightarrow W_n\cO_X(D),
\end{alignat*}
induced respectively by 
\[
\begin{aligned}
F &\colon W_{n+1}\cK_X \longrightarrow F_*W_n\cK_X, \\
V &\colon F_*W_n\cK_X \longrightarrow W_{n+1}\cK_X, \\
R &\colon W_{n+1}\cK_X \longrightarrow W_n\cK_X,
\end{aligned}
\]
where for the Frobenius map, see \cref{Frobenius-grade}.
\item 
We have the following exact sequence (cf.~\cite{tanaka22}*{Proposition~3.7}): 
\begin{equation} \label{eq:key-sequence-for-WnI}
0 \longrightarrow F_*^nW_m\cO_X(p^nD) 
\xrightarrow{V^n} W_{n+m}\cO_X(D) 
\xrightarrow{R^m} W_n\cO_X(D) 
\longrightarrow 0.
\end{equation}
\end{enumerate}

\section{Foundations of quasi-\texorpdfstring{$F$}{F}-splitting}
The purpose of this section is to develop the basic framework of quasi-$F$-splitting in mixed characteristic.
We introduce its definition, examine equivalent formulations, and relate it to local cohomology and vanishing results.

\subsection{Definition and basic properties of  quasi-$F$-splitting}
We define the modules $Q_{X,\Delta,n}$ and the morphisms $\Phi_{X,\Delta,n}$ via pushout diagrams,
and establish their fundamental exact sequences.
These constructions generalize the $F$-splitting framework to mixed characteristic.

\begin{definition} 
Let $X$ be a normal Noetherian separated scheme such that $X$ is flat over $\Spec \Z_{(p)}$.
Let $\Delta$ be a $\bQ$-divisor on $X$.
We define $Q_{X,\Delta,n}$ and $\wt{\Phi_{X, \Delta, n}}$ by the following pushout diagram of $W_n\cO_X$-module homomorphisms: 
\begin{equation} \label{diagram:quasi-F-split-definition}
\begin{tikzcd}
W_{n+1}\cO_X(\Delta) \arrow{r}{F} \arrow{d}{R^{n}} & F_* W_n \cO_X(p\Delta)  \arrow{d}\\
\cO_X(\Delta) \arrow{r}{\wt{\Phi_{X, \Delta, n}}}& \arrow[lu, phantom, "\usebox\pushoutdr" , very near start, yshift=0em, xshift=0.6em, color=black] Q_{X, \Delta, n}.
\end{tikzcd}
\end{equation} 
\end{definition}

\begin{proposition}
Let $X$ be a normal Noetherian separated scheme such that $X$ is flat over $\Spec \Z_{(p)}$.
Let $\Delta$ be a $\bQ$-divisor on $X$.
Then we have the following two exact sequences:
\begin{align}
    0 &\to F_*W_n\cO_X(p\Delta) \xrightarrow{\cdot p} F_*W_n\cO_X(p\Delta) \to Q_{X,\Delta,n} \to 0, \label{eq:exact-p}\\
    0 &\to F^n_*Q_{X,p^n\Delta,m} \xrightarrow{V^n} Q_{X,\Delta,n+m} \xrightarrow{R^m} Q_{X,\Delta,n} \to 0. \label{eq:exact-V-R}
\end{align}
% \commentbox{$F_*W_n\cO_X(p\Delta)$ ?}
\end{proposition}

\begin{proof}
Since $\cO_X$ is $p$-torsion free, so is $W_n\cK_X$.
Therefore, $W_n\cO_X(\Delta)$ is also $p$-torsion free.
By the definition of $Q_{X,\Delta,n}$ and (\ref{eq:key-sequence-for-WnI}), we obtain the exact sequence
\[
0 \to F_*W_n\cO_X(p\Delta) \xrightarrow{FV} F_*W_n\cO_X(p\Delta) \to Q_{X,\Delta,n} \to 0.
\]
Since $FV=p$ on $W_n\cK_X$, the same holds on $W_n\cO_X(p\Delta)$.
Thus we obtain the exact sequence \eqref{eq:exact-p}.
Finally, \eqref{eq:exact-V-R} follows from \eqref{eq:key-sequence-for-WnI} together with \eqref{eq:exact-p}.
\end{proof}

\begin{definition}
Let $X$ be a normal Noetherian separated scheme such that $X$ is flat over $\Spec \Z_{(p)}$.
Let $\Delta$ be a $\bQ$-divisor on $X$.
By (\ref{eq:exact-p}), the map $\wt{\Phi_{X,\Delta,n}}$ induced $\cO_{\var{X}}$-module homomorphism
\[
\Phi_{X,\Delta,n} \colon \cO_{X}(\Delta)|_{\var{X}} \to Q_{X,\Delta,n}
\]
and $Q_{X,\Delta,n}$ has a natural $\cO_{\var{X}}$-module via $\Phi_{X,\Delta,n}$.
We denote $\Phi_{X,0,n}$ by $\Phi_{X,n}$.
\end{definition}

\begin{definition}
\label{definition:log-quasi-F-split}
Let $X$ be a normal Noetherian separated scheme such that $X$ is flat over $\Spec \Z_{(p)}$ and $X_{p=0}$ is $F$-finite. 
\begin{itemize}
    \item We say that $X$ is \emph{$n$-quasi-$F$-split} if the homomorphism $\Phi_{X,n}$ splits as a homomorphism of $\cO_{\var{X}}$-modules.
    \item We say that $X$ is \emph{quasi-$F$-split} if it is $n$-quasi-$F$-split for some $n \in \Z_{>0}$.
    \item We define the \emph{quasi-$F$-splitting height} $\sht(X)$ by
    \[
    \sht(X):=\min \{\,n \in \Z_{>0} \mid X \text{ is $n$-quasi-$F$-split}\,\},
    \]
    if $X$ is quasi-$F$-split, and set $\sht(X)=\infty$ otherwise.
\end{itemize}
\end{definition}

\begin{notation}\label{notation:relative}
Let $(B,\m)$ be an excellent local ring with $p \in \m$ and a dualizing complex $\omega_B^\bullet$. 
Let $X$ be a $(d+1)$-dimensional integral normal scheme which is projective over $\Spec B$ for  $d \geq 1$ and $\omega_{X}^{\bullet}:=(X \to \Spec B)^! \omega_B^{\bullet}$. 
We assume that $X$ is $p$-torsion free and $\var{X}:=X_{p=0}$ is $F$-finite, normal and  integral.
%\commentbox{teppei: $d+1 \geq 2$?}
\end{notation}

% \begin{remark}
% \commentbox{Teppei: iru? $\cdots$ \cite[3.15]{KTTWYY1} de tsukau kedo...}
% Let $X$ be a Noetherian normal scheme in characteristic $p>0$.
% Then there is an exact sequence
% \[
% 0 \to  
% F_*B_{X, p\Delta, n} \to Q_{X,\Delta,n+1} \to F_* \cO_X(p\Delta) \to 0,
% \]
% where 
% \[
% B_{X, \Delta, n} := \operatorname{Coker}\bigl( 
% W_n\cO_X(\Delta) \xrightarrow{F} F_*W_n\cO_X(p\Delta) 
% \bigr).
% \]
% {\cred See \cite[Lemma 3.8]{KTTWYY1} for details.}
% \end{remark}

\begin{proposition}\label{efs-to-another-split}
Let notation be as in \cref{notation:relative}. 
Let $D$ be a Weil divisor on $X$ such that $D$ and $\mathrm{div}(p)$ have no common component.
Assume further that $D$ is Cartier in codimension two around $\var{X}$. 
Then the following are equivalent:
\begin{enumerate}
    \item $\Phi_{X,n}$ splits;
    \item $\Phi^{**}_{X,D,n}$ splits,
\end{enumerate}
where 
\[
\Phi^{**}_{X,D,n} \colon \cO_{\var{X}}(\var{D})=\cO_X(D)|_{\var{X}}^{**} \to Q_{X,D,n}^{**}
\]
is the reflexive hull of $\Phi_{X,D,n}$ on $\var{X}$.
\end{proposition}

\begin{proof}
Let $U \subseteq X$ be the Cartier locus of $D$.
%Since $\var{D}:=D|_{\var{X}} \subset \overline{X}$ is Cartier in codimension one, the open subset $U$ contains all codimension-two points of $X$. 
%\commentbox{teppei: Why is $D|_{X\setminus \overline{X}}$ Cartier in Codimension 2? "around $\overline{X}$"?}
We note that
\[
\cO_X(D)|_{\var{X}}^{**}=j_*\cO_{U}(D)|_{\var{X}}=\cO_{\var{X}}(\var{D}),
\]
where $j \colon U \to \var{X}$ is the open immersion.
By \eqref{eq:exact-p}, we have
\[
(Q_{X,n} \otimes_{\cO_X} \cO_X(D))|_{U} \simeq Q_{U,D,n}.
\]
Hence,
\[
\Hom_{\cO_{\var{X}}}(Q_{X,n},\cO_{\var{X}}) 
\simeq 
\Hom_{\cO_{\var{X}}}(Q_{X,D,n}^{**},\cO_{\var{X}}(\var{D})),
\]
since this is an isomorphism on $U \cap \var{X}$, which contains all the codimension-one points of $\var{X}$.  
The desired equivalence follows.
\end{proof}

\begin{proposition}\label{efs-local-coh}
Let notation be as in \cref{notation:relative}. 
Then $X$ is $n$-quasi-$F$-split if and only if the map
\[
\Phi_{X,K_X,n}^{**} \colon H^d_\m(\omega_{\var{X}}) \longrightarrow H^d_\m(Q_{X,K_X,n}^{**})
\]
is injective.
\end{proposition}

\begin{proof}
We may assume that $K_X$ has no common component with $\var{X}$.
%\commentbox{teppe: Need to assume $\var{X}$: integral?}
%We further assume that $K_X|_{\var{X}}$ is Cartier in codimension one.
Since $\var{X}$ is Gorenstein in codimension one, we have $K_{\var{X}}:=K_X|_{\var{X}}$ is a canonical divisor.
By \cref{efs-to-another-split}, the $n$-quasi-$F$-splitting of $X$ is equivalent to the surjectivity of the evaluation map
\[
\Hom_{\cO_{\var{X}}}(Q_{X,K_X,n}^{**},\omega_{\var{X}}) 
\longrightarrow 
\Hom_{\cO_{\var{X}}}(\omega_{\var{X}},\omega_{\var{X}}) 
\simeq H^0(\cO_{\var{X}}).
\]
Taking Matlis duals, this surjectivity is equivalent to the injectivity of
\[
H^d_\m(\omega_{\var{X}}) \longrightarrow H^d_\m(Q_{X,K_X,n}^{**}),
\]
as desired.
\end{proof}

\begin{proposition}\label{Serre-cond}
Let notation be as in \cref{notation:relative}.
Let $\Delta$ be a $\bQ$-Weil divisor on $X$, and $i \in \Z$.
If $H^i_\m(\cO_X(p^e\Delta)|_{\var{X}})=0$ for every $e \geq 1$, then 
\[
H^i_\m(Q_{X,p^e\Delta,n})=0 \quad \text{for all } e \geq 0 \text{ and } n \geq 1.
\]
\end{proposition}

\begin{proof}
Note that, when $n=1$, this follows from the assumption since we have $Q_{X,p^e\Delta,1} \simeq F_{*} \cO_{X} (p^{e+1} \Delta)|_{\var{X}}$ by definition.
The remaining cases follow by induction on $n$ using \eqref{eq:exact-V-R}.
%the exact sequence \eqref{eq:exact-V-R} together with an inductive argument on $n$.
%\commentbox{teppei: $e=0$ case??}
\end{proof}

\subsection{Kawamata--Viehweg-type vanishing}
In this subsection, we prove Kawamata--Viehweg type vanishing theorems
under the assumption of quasi-$F$-splitting and Cohen–Macaulayness.
This provides the key cohomological application of our theory.

\begin{theorem}\label{thm:KV-vanishing-dual}
Let notation be as in \cref{notation:relative}.
Assume that $X$ is Cohen-Macaulay and quasi-$F$-split. 
Let $A$ be an ample $\Q$-Cartier Weil divisor.
Suppose that $\cO_X(-p^eA)$ is Cohen–Macaulay for every $e \geq 0$.
Then
\[
H^j_\m(X,\cO_X(-A)) = 0
\quad \text{for all } j<d+1.
\]
\end{theorem}

\begin{proof}
We may assume $A$ has no common component of $\mathrm{div}(p)$ and set $\var{A}:=A|_{\var{X}}$.
%\commentbox{teppe: Need to assume $\var{X}$: integral?}
We first show $H^j_\m(\var{X},\cO_X(-A)|_{\var{X}})=0$ for $j<d$.
Fix $j<d$.
We note that since $\cO_X(-p^eA)$ is Cohen-Macaulay, we have $\cO_X(-p^eA)|_{\var{X}}=\cO_{\var{X}}(-p^e\var{A})$ for every $e \geq 0$.
In particular, $\var{A}$ is ample $\Q$-Cartier.
Fix $n \geq 1$ such that $X$ is $n$-quasi-$F$-split.
Since $\cO_X(-p^eA)|_{\var{X}}$ is Cohen–Macaulay for every $e \geq 0$, there exists $m_0 \geq 1$ such that  we have $H^j_\m(\cO_{\var{X}}(-p^m\var{A}))=0$  for $m \geq m_0$ by the proof of \cite{KTTWYY1}*{Theorem~3.15}.
Then we have $H^j_\m(Q_{X,-p^{m_0-1}A,n})=0$ by \cref{Serre-cond}.
Since $\Phi_{X,-p^{m_0-1}A,n}$ splits by \cref{efs-to-another-split}, we have $H^j_\m(\cO_{\var{X}}(-p^{m-1}\var{A}))=0$.
Repeating such a process, we obtain $H^j_\m(\var{X},\cO_{\var{X}}(-\var{A}))=0$.

Now consider the exact sequence
\[
0 \to \cO_X(-A) \xrightarrow{\cdot p} \cO_X(-A) \to \cO_{\var{X}}(-\var{A}) \to 0.
\]
We deduce $H^j_\m(\cO_X(-A))=0$ for $j<d$, and $H^d_{\m}(\cO_X(-A))$ is $p$-torsion free.
%Therefore,
Since $p\in \m$, we have
\[
H^j_\m(\cO_X(-A))=0
\]
for $j<d+1$, as desired.
\end{proof}

\begin{theorem}[cf.~Theorem\ref{intro:thm:KV-vanishing},\cite{KTTWYY1}*{Theorem~3.15}]\label{thm:KV-vanishing}
Let notation be as in \cref{notation:relative}.
%Assume that  $B/p$ is an Artin ring.
Assume that $X$ is quasi-$F$-split and $X$ is Cohen-Macaulay.
%, $X_{\Q}$ is log canonical, and $X$ is Cohen-Macaulay. 
Let $L$ be a $\Q$-Cartier Weil divisor such that $A:=L-K_X$ is ample. 
Suppose that $\cO_X(-p^eA)$ is Cohen–Macaulay for every $e \geq 0$.
Then
\[
H^j(X,\cO_X(L)) = 0
\quad \text{for all } j>0.
\]
\end{theorem}

\begin{proof}
By \cref{thm:KV-vanishing-dual}, we have
\[
H^j_\m(X, \cO_X(-A)) = 0
\]
for all \( j < d+1 \).
Hence
\[
R\Gamma_\m(\cO_X(-A)) \simeq H^{d+1}_\m(\cO_X(-A))[-(d+1)].
\]
By Grothendieck duality, we obtain
\[
R\Gamma(X, \cO_X(L))
  \simeq R\cHom_{R}\bigl(R\Gamma(\cO_X(K_X - L)), \omega_B^\bullet\bigr)[-d-1].
\]
By Matlis duality, the right-hand side is the Matlis dual of
\[
R\Gamma_\m(X, \cO_X(K_X - L))[d+1].
\]
Since \(K_X - L = -A\), it follows that
\[
R\Gamma(X, \cO_X(L)) \simeq H^0(X, \cO_X(L))[0],
\]
as desired.
\end{proof}

\begin{remark}
\cref{thm:KV-vanishing} provides a mixed characteristic analogue of \cite{KTTWYY1}*{Theorem~3.15}.
We note that \cite{KTTWYY1}*{Theorem~3.15} requires the local assumption that $X$ is $F$-pure.
\end{remark}

\subsection{A cone correspondence for quasi-$F$-splitting}
We establish a correspondence between the quasi-$F$-splitting of a projective scheme
and that of its homogeneous coordinate ring.
This “cone correspondence’’ allows us to reduce geometric questions to algebraic ones.

\begin{theorem}\label{cone-corr}
Let notation be as in \cref{notation:relative}.
% Assume that $B/p$ is an Artin ring and that $d \geq 1$.
Assume that $B/p$ is an Artin ring.
Let $L$ be an ample Cartier divisor on $X$ and set $R:=\bigoplus_{m \geq 0}H^0(\cO_X(mL))$.
Then we have $\sht(R)=\sht(X)$.   
%\commentbox{teppei:Does $R$ satisfy \cref{notation:relative}? $R$ is not local.}
\end{theorem}

\begin{proof}
Let $\n:=(R_+,p)$.
We may assume that $L$ and $K_X$ have no common component with $\mathrm{div}(p)$.
Note that $\Proj(R) \simeq X$ and $\widetilde{R(m)} \simeq \cO_{X}(mL)$.
Since $X$ is Gorenstein in codimension two and $p$-torsion free around $p=0$, the same holds for $R$.
%\commentbox{teppei:around $p=0$?}

Set $U:=\Spec R \setminus V(R_+)$ and let $\pi \colon U \to X$ denote the natural projection.
Since $\pi$ is smooth and its relative differential sheaf is trivial, we have
\[
\omega_U=h^{-d-2}(\pi^{!}\omega_{X}^{\bullet}) \simeq h^{-d-1}(\pi^*\omega_X^{\bullet})=\pi^*\omega_X.
\]
Therefore,
\[
\omega_R \simeq j_*\omega_U = j_*\pi^*\omega_X = \bigoplus_{ m \in \Z} H^0(X,\cO_X(K_X+mL)),
\]
where $j \colon U \to \Spec R$ is the open immersion.
Hence there exists a canonical divisor $K_R$ on $R$ such that
\[
R(K_R)=\bigoplus_{ m \in \Z} H^0(X,\cO_X(K_X+mL)) \subseteq \Frac(R).
\]
By the same argument as in \cite{KTTWYY2}*{Section~7}, we obtain 
\[
W_n(R)(K_R)=\bigoplus_{ m \in \Z} H^0(X,W_n\cO_X(K_X+mL))
\]
for every $n \geq 1$.
In particular, we have isomorphisms
\[
\bigoplus_{ m \in \Z}H^{d}(X,\omega_X(mL)) \simeq H^{d+1}_{R_+}(\omega_R), 
\qquad 
\bigoplus_{m \in \Z}H^{d}(X,W_n\cO_X(pK_X+mL)) \simeq H^{d+1}_{R_+}(W_n(R)(pK_R)),
\]
which commute with restriction maps and Frobenius homomorphisms.
Taking modulo $p$, we obtain isomorphisms
\begin{equation}\label{eq:isom}
    \bigoplus_{m \in \Z}H^{d}(\var{X},\omega_{\var{X}}(m\var{L})) \simeq H^{d+1}_{\n}(\omega_{\var{R}}),
    \qquad
    \bigoplus_{m \in \Z}H^{d}(X,Q_{X,K_X+mL,n}^{**}) \simeq H^{d+1}_{\n}(Q_{R,n}\otimes \omega_{\var{R}}),
\end{equation}
which commute with $\Phi_{R,K_R,n}$ and $\Phi_{X,K_X+mL,n}$.
%\commentbox{teppei: $m \in \Z$?}

If $R_{\n}$ is $n$-quasi-$F$-split, then we obtain an injection 
\[
H^{d+1}_{\n}(\omega_{\var{R}}) \hookrightarrow H^{d+1}_{\n}(Q_{R,n}\otimes \omega_{\var{R}}).
\]
By \eqref{eq:isom}, this yields the injection
\[
\bigoplus_{m \geq 0}H^{d}(\var{X},\omega_{\var{X}}(m\var{L})) \hookrightarrow \bigoplus_{m \geq 0}H^{d}(X,Q_{X,K_X+mL,n}^{**}),
\]
so $X$ is $n$-quasi-$F$-split by \cref{efs-local-coh}.

Conversely, if $X$ is $n$-quasi-$F$-split, then by \cref{efs-to-another-split} we obtain the injection
\[
\bigoplus_{m \geq 0}H^{d}(\var{X},\omega_{\var{X}}(m\var{L})) \hookrightarrow \bigoplus_{m \geq 0}H^{d}(X,Q_{X,K_X+mL,n}^{**}).
\]
Thus, by \eqref{eq:isom}, $R_\n$ is $n$-quasi-$F$-split.
Therefore, $\sht(R_\n)=\sht(X)$.
The desired result follows from \cref{graded-localization}.
\end{proof}

\begin{proposition}\label{graded-localization}
Let $R$ be a graded Noetherian ring with a dualizing complex $\omega_R^{\bullet}$ such that $(R_0,\n_0)$ is local and $\var{R}$ is $F$-finite.
Then $\sht(R)=\sht(R_\n)$, where $\n=(\n_0,R_+)$.
%\commentbox{teppei: does $R$ satisfy \cref{notation:relative}?}
\end{proposition}

\begin{proof}
By the same argument as in \cite{KTTWYY2}*{Proposition~7.1}, the ring $W_n(R)$ has a natural graded structure, and so does $Q_{R,n}$.
Furthermore, the homomorphism $\Phi_{R,n} \colon \var{R} \to Q_{R,n}$ is graded.
By the same argument as in \cite{kty}*{Proposition~2.25}, we obtain the desired result.
\end{proof}

\section{RDP del Pezzo surfaces}
The aim of this section is to apply the general theory of quasi-$F$-splitting
to rational double point (RDP) del Pezzo surfaces.
We prove that every mixed characteristic lift of an RDP del Pezzo surface is quasi-$F$-split,
and we analyze several explicit examples of hypersurface singularities.

% \begin{notation}\label{notation:RDP-dP}
% Let $\var{X}$ be a RDP del Pezzo surface over an algebraically closed field $k$ of characteristic $p>0$.
% Let $X$ be a lift of $\var{X}$ over $W(k)$.
% \end{notation}

\begin{definition}
Let $k$ be an algebraically closed field of characteristic $p>0$, and $S$ a normal projective surface over $k$.
We say $S$ is an \emph{RDP del Pezzo surface} over $k$ if $X$ has at worst rational double points and $-K_{X}$ is ample.
%\commentbox{an RDP ?  a RDP?}
\end{definition}

\begin{definition}
Let $k$ be an algebraically closed field of characteristic $p>0$, and $\var{X}$ an RDP del Pezzo surface over $k$.
\begin{enumerate}
    \item 
    A \emph{lift of $\var{X}$ over $W(k)$} is a flat projective scheme $X$ over $W(k)$ such that $X_k \simeq \var{X}$. If there is a lift of $\var{X}$ over $W(k)$, we say $\var{X}$ is \emph{$W(k)$-liftable}. 
    \item 
    If there exists a log resolution $(\var{Y}, \var{E}) \rightarrow X$ and a log smooth projective pair $(Y, E)$ (that is, $Y$ is smooth projective scheme over $W(k)$ and $E \subset Y$ is relative simple normal crossing divisor) over $W(k)$ such that $(Y, E)_{k} \simeq (\var{Y}, \var{E})$, we say $\var{X}$ is \emph{$W(k)$-log liftable}.
\end{enumerate}
\end{definition}

\begin{theorem}[{Theorem \ref{intro:RDP-dP}}]
\label{RDP-dP}
Let $\var{X}$ be an RDP del Pezzo surface over an algebraically closed field $k$ of characteristic $p>0$.
Let $X$ be a lift of $\var{X}$ over $W(k)$.
Then $X$ is quasi-$F$-split.
\end{theorem}

\begin{proof}
If $\var{X}$ is $W(k)$-log liftable, then $\var{X}$ is quasi-$F$-split by \cite{KTTWYY1}*{Theorem~6.3}, and hence so is $X$ by \cite{Yoshikawa25}*{Proposition~4.6}.
Thus we may assume that $\var{X}$ is not $W(k)$-log liftable.
By \cite{KN22}*{Theorem~1.7~(1)}, we have
\[
(p,\mathrm{Dyn}(X))=(2,7A_1),(2,8A_1),(2,4A_1+D_4),(3,4A_2),
\]
where $\mathrm{Dyn}(X)$ is %the Dynkin type of $X$, i.e., 
the corresponding Dynkin diagrams of singularities on $X$.
Since $H^1(\cO_{\var{X}}(-mK_{\var{X}}))=0$ and $H^1(\cO_{X_{\Q}}(-mK_{X_{\Q}}))=0$ for every $m \geq 0$, it follows that $H^1(\cO_X(-mK_X))=0$.
Therefore, we obtain the exact sequence of graded rings
\begin{equation}\label{eq:R-flar}
    0 \to R:=\bigoplus_{m \geq 0} H^0(\cO_X(-mK_X)) 
    \xrightarrow{ \cdot p}
    R \to \var{R}:=\bigoplus_{m \geq 0} H^0(\cO_{\var{X}}(-mK_{\var{X}})) \to 0.
\end{equation}
In particular, each $R_m$ is a free $W(k)$-module of rank $h^0(\cO_{\var{X}}(-K_{\var{X}}))$.
By the same argument as in \cite{KN23}*{Propositions~3.14, 3.15, 3.23}, there exists a graded surjective homomorphism
\[
\varphi \colon W(k)[x,y,z,w] \twoheadrightarrow R,
\]
such that $\varphi$ induces the surjective graded homomorphism
\[
\bar{\varphi} \colon k[x,y,z,w] \twoheadrightarrow \var{R}
\]
with $\Ker(\bar{\varphi})=(\bar{f})$ for some homogeneous element $\bar{f} \in k[x,y,z,w]$, where $W(k)[x,y,z,w]$ is a weighted polynomial ring.
In particular, there exists a homogeneous element $f \in W(k)[x,y,z,w]$ of degree $\deg(\bar{f})$ such that $f$ lifts $\bar{f}$  and $f \in \Ker(\varphi)$.
%\commentbox{teppei: Should we take a lift $f \in W(k)$ such that $f \in \ker \varphi$?}
By \eqref{eq:R-flar}, we have $W(k)[x,y,z,w]/(f) \simeq R$.
Thus $X$ is a hypersurface in a weighted projective space over $W(k)$.
By \cref{cone-corr}, it suffices to show that $W(k)[x,y,z,w]/(f)$ is quasi-$F$-split.
By the proof of \cite{KN23}*{Propositions~3.14, 3.15, 3.23}, we may assume that $\var{f}$ is given by \cite{KN23}*{(3.14.2), (3.15.4), (3.15.5), or (3.23.7)}.
Therefore, the result follows from \cref{7A_1,8A_1,4A_1+D_4,4A_2}.
%Therefore, the result follows from \cite{KN23}*{Propositions~3.14, 3.15, 3.23} and \cref{7A_1,8A_1,4A_1+D_4,4A_2}.
%\commentbox{teppei: How to assume that $\var{f}$ is as in \cite{KN23}?}
\end{proof}

\subsection{Quasi-$F$-splitting for explicit hypersurfaces}
We study explicit weighted hypersurface equations corresponding to the RDP types $(7A_1)$, $(8A_1)$, $(4A_1+D_4)$, and $(4A_2)$,
and compute their quasi-$F$-splitting heights by the Fedder-type criterion \cite[Theorem~4.13]{Yoshikawa25}.
We will use the criterion freely in the subsection.

\begin{notation}\label{notation:fedder}
Let $k$ be an algebraically closed field of characteristic $p>0$.
Let $W(k)[x,y,z,w]$ be a weighted polynomial ring and $f \in W(k)[x,y,z,w]$ a homogeneous element. 
For $G \in W(k)[x,y,z,w]$, the coefficient of $x^{i_1}y^{i_2}z^{i_3}w^{i_4}$ is denoted by $G_{i_1i_2i_3i_4}$.
We use notation in \cite{Yoshikawa25}*{Example~4.14}.
\end{notation}

\begin{proposition}[$7A_1$]\label{7A_1}
Let notation be as in \cref{notation:fedder} and assume
\[
(p,\deg(x),\deg(y),\deg(z),\deg(w))=(2,1,1,1,2)
\]
and
\[
f=w^2+xyz(x+y+z).
\]
Then the following hold:
\begin{enumerate}
    \item The ring $W(k)[x,y,z,w]/(f+pG)$ is $3$-quasi-$F$-split for every homogeneous element $G$ of degree $4$. 
    \item $\sht(W(k)[x,y,z,w]/(f))=2$.
    \item $\sht(W(k)[x,y,z,w]/(f+p(xy+yz+xz)w))=3$.
\end{enumerate}
\end{proposition}

\begin{proof}
Let $G$ be a homogeneous element of degree $4$.
We note that 
\[
G \equiv G_{1101}xyw+G_{1011}xzw+G_{0111}yzw \pmod{\m^{[2]}}.
\]
Then
\[
f\Delta_1(f+pG) \equiv x^2y^2z^2(xy+yz+xz)w^2 + fG^2 \pmod{\m^{[4]}}.
\]
This yields assertions (2) and (3).

For (1), assume for contradiction that 
\[
R:=W(k)[x,y,z,w]/(f+pG)
\]
is not $3$-quasi-$F$-split.
Then necessarily
\[
G_{1101}=G_{1011}=G_{0111}=1.
\]
Since $\sht(R) \geq 4$, the coefficient of $x^4y^5z^7w^6$ in $f(\Delta_1(f+pG))^3$,  which is equal to
\[
G_{0111}^4G_{1011}^2 + G_{0021}^2 +G_{0111}^2 +1
\]
must vanish.
Therefore, we obtain $G_{0021} =1$.
By symmetry, we also obtain $G_{2001} = G_{0201} =1$.
On the other hand, the coefficient of $x^5y^5z^6w^6$ in $f(\Delta_1(f+pG))^3$ is
\[
G_{1101}^4G_{0021}^2 + G_{1011}^4G_{0201}^2 + G_{0111}^4G_{2001}^2 + G_{1011}^2 + G_{0111}^2 + G_{1101}^2
+1 = 1,
\]
so we obtain the contradiction.
% \commentbox{(onuki) $G_{0021} = 1$ ? Then we obtain contradiction from $x^5y^5z^6w^6$.}
% \commentbox{(Takamatsu) YES!! Thank you!! I have updated it.}
% % must vanish.
% % Therefore, we obtain $G_{0021} =0$.
% %A direct computation shows $G_{0021}=0$.
% %By symmetry, we also obtain $G_{2001}=G_{0201}=0$.
% % {\cred Moreover, the coefficient of $x^3y^6z^7w^6$ in $f(\Delta_1(f+pG))^3$ is 
% % \[
% % G_{0111}^4G_{1011}^2 + G_{0111}^4 + G_{0021}^2 + G_{0111}^2 =1.
% % \]
% % Therefore, we obtain the contradiction.
% %}
% \commentbox{↓: Need $+ G_{1011}^2 + G_{0111}^2 + G_{1101}^2 $?}
% Moreover, the coefficient of $x^5y^5z^6w^6$ in $f(\Delta_1(f+pG))^3$ is
% \[
% G_{1101}^4G_{0021}^2 + G_{1011}^4G_{0201}^2 + G_{0111}^4G_{2001}^2 
% +1 = 0,
% \]
% %which is nonzero.
% This contradiction shows that $\sht(R) \leq 3$, as claimed.
\end{proof}

\begin{proposition}[$8A_1$]
\label{8A_1}
Let notation be as in \cref{notation:fedder} assume
\[
(p,\deg(x),\deg(y),\deg(z),\deg(w))=(2,1,1,2,3)
\]
and
\begin{align*}
   f&= w^2+[abc]z^3+[(ab+bc+ca)^2+abc(a+b+c)]y^2z^2 \\
   &+[(a+b+c)(a+b)(b+c)(a+c)]xyz^2 \\
   &+[(ab+bc+ca)^2]x^2z^2+[(a+b+c)^2(a+b)(b+c)(c+a)]xy^3z \\
   &+[(a+b+c)^2((a+b+c)^3+abc)]x^2y^2z+[(a+b+c)^2(a+b)(b+c)(c+a)]x^3yz \\
   &+[(a+b+c)^2abc]x^4z+[(a+b)^2(b+c)^2(b+c)^2(c+a)^2]y^6 \\
   &+[(a+b+c)^3+abc]^2x^2y^4+[(a+b)^2(b+c)^2(c+a)^2]x^4y^2+[a^2b^2c^2]x^6
\end{align*}
for some $a,b,c \in k$ with $abc(a+b)(b+c)(c+a)(a+b+c) \neq 0$.
\begin{enumerate}
    \item The ring $W(k)[x,y,z,w]/(f+pG)$ is $3$-quasi-$F$-split for every homogeneous element $G$ of degree $6$.
    \item $\sht(W(k)[x,y,z,w]/(f))=2$.
    \item $\sht(W(k)[x,y,z,w]/(f+p(a+b+c)^2 yzw))=3$.
\end{enumerate}
\end{proposition}

\begin{proof}
We set
\[
s_1=a+b+c,\quad s_2=ab+bc+ca,\quad s_3=abc,
\]
so that $s_1,s_2 \neq 0$ and
\[
s_1s_2+s_3=(a+b)(b+c)(c+a) \neq 0.
\]
First, we show the assertion (1).
Suppose, for contradiction, that
\[
R:=W(k)[x,y,z,w]/(f+pG)
\]
is not $3$-quasi-$F$-split.
Since $R$ is not $2$-quasi-$F$-split, we obtain
\begin{equation}
\label{eqn:8AfDelta}
f(\Delta_1(f)+G^2) \equiv s_1^2(s_1s_2+s_3)\bigl(s_1^4+G_{1011}^2+G_{0111}^2\bigr)x^3y^3z^3w^2 \in \m^{[4]}.
\end{equation}
As $s_1,s_1s_2+s_3 \neq 0$, this yields $G_{1011}=G_{0111}+s_1^2$. 
From now on, we eliminate $G_{1011}$ using this relation.

Next, since $R$ is not $3$-quasi-$F$-split, we have $f\Delta_1(f+pG)^3 \in \m^{[8]}$.
The coefficient of $x^7y^3z^7w^6$ in $f\Delta_1(f+pG)^3$ is
\[
s_1^2(s_1s_2+s_3)\bigl((s_1^4s_2^2+s_1^2s_3^2)G_{0111}^2+s_3^2G_{3001}^2+s_3^2G_{2101}^2+s_1^8s_2^2+s_1^4s_2^4+s_1^6s_3^2\bigr),
\]
which must vanish.
Hence
\begin{equation}\label{eq:G2101}
G_{2101}^2
=G_{3001}^2+\frac{(s_1^4s_2^2+s_1^2s_3^2)G_{0111}^2+s_1^8s_2^2+s_1^4s_2^4+s_1^6s_3^2}{s_3^2}.
\end{equation}

% Similarly, the coefficient of $x^3y^3z^7w^6$ in $f\Delta_1(f+pG)^3$ is
Similarly, the coefficient of $x^3y^7z^7w^6$ in $f\Delta_1(f+pG)^3$ is
\[
s_1^2(s_1s_2+s_3)\bigl((s_1^4s_2^2+s_1^2s_3^2)G_{0111}^2+s_3^2G_{1201}^2+s_3^2G_{0301}^2+s_1^4s_2^4+s_1^6s_3^2\bigr),
\]
which must vanish.
Thus
\begin{equation}\label{eq:G1201}
G_{1201}^2
=G_{0301}^2+\frac{(s_1^4s_2^2+s_1^2s_3^2)G_{0111}^2+s_1^4s_2^4+s_1^6s_3^2}{s_3^2}.
\end{equation}

% Now, consider the coefficient of $x^3y^7z^7w^6$ in $f\Delta_1(f+pG)^3$.
Now, consider the coefficient of $x^5y^5z^7w^6$ in $f\Delta_1(f+pG)^3$.
After dividing by $s_1^2(s_1s_2+s_3)$, substituting \eqref{eq:G2101} and \eqref{eq:G1201}, we obtain
\begin{equation}\label{eq:G0111}
\left(s_1^2G_{0111}^2+s_3G_{3001}+s_3G_{0301}+(s_1^4+s_1s_3)G_{0111}+s_1^2s_2^2\right)^2=0
\end{equation}
% \commentbox{\eqref{eq:G0111}: before dividing by $s_1^2(s_1s_2+s_3)$.}
% \begin{align}\label{eq:G0111}
% &s_{1}^{11}s_{2}G_{0111}^{2} + s_{1}^{10}s_{3}G_{0111}^{2} + s_{1}^{7}s_{2}^{5} 
% + s_{1}^{7}s_{2}G_{0111}^{4} + s_{1}^{6}s_{2}^{4}s_{3} + s_{1}^{6}s_{3}G_{0111}^{4} \\
% &\quad+ s_{1}^{5}s_{2}s_{3}^{2}G_{0111}^{2} + s_{1}^{4}s_{3}^{3}G_{0111}^{2} 
% + s_{1}^{3}s_{2}s_{3}^{2}G_{0301}^{2} + s_{1}^{3}s_{2}s_{3}^{2}G_{3001}^{2} \notag \\
% &\quad+ s_{1}^{2}s_{3}^{3}G_{0301}^{2} + s_{1}^{2}s_{3}^{3}G_{3001}^{2}, \notag
% \end{align}
% which must vanish.

Finally, the coefficient of $x^7y^7z^5w^6$ in $f\Delta_1(f+pG)^3$, after dividing by $s_1^2(s_1s_2+s_3)$, eliminating $G_{2101}^2,$ $G_{1201}^2,$ and $G_{0301}^2$ using \eqref{eq:G2101}, \eqref{eq:G1201}, and \eqref{eq:G0111}, is
\[
s_1^{10}s_2^2 + s_1^8s_3^2
= s_1^{8}(s_1s_2+s_3)^{2},
\]
% \commentbox{below: before dividing by $s_1^2(s_1s_2+s_3)$.}
% \[
% s_{1}^{13}s_{2}^{3} + s_{1}^{12}s_{2}^{2}s_{3}
% + s_{1}^{11}s_{2}s_{3}^{2} + s_{1}^{10}s_{3}^{3}
% = s_1^{10}(s_1s_2+s_3)^{3},
% \]
% which is a unit.
%a contradiction.
Therefore, we obtain a contradiction as desired. The assertions
(2) and (3) follow from (\ref{eqn:8AfDelta}).
\end{proof}

\begin{proposition}[$4A_1+D_4$]\label{4A_1+D_4}
Let notation be as in \cref{notation:fedder} assume
\[
(p,\deg(x),\deg(y),\deg(z),\deg(w))=(2,1,1,2,3)
\]
and
\[
f=w^2+z^3+[a]x^2z^2+y^4z+[a^2+a+1]x^2y^2z+[a(a+1)]x^3yz
\]
for some $a \in k$ with $a(a+1) \neq 0$.

\begin{enumerate}
    \item The ring $W(k)[x,y,z,w]/(f+pG)$ is $4$-quasi-$F$-split for every homogeneous element $G$ of degree $6$.
    \item $\sht(W(k)[x,y,z,w]/(f+ pyzw))=2$.
    \item $\sht(W(k)[x,y,z,w]/(f))=3$.
    \item $\sht(W(k)[x,y,z,w]/(f+p(a(1+a)x^3y^3 + (1+a)xzw))=4$.
\end{enumerate}
\end{proposition}

\begin{proof}
First, we show the assertion (1).
We note that since the degree of $G$ is $6$, we have
\[
G\equiv G_{1011}xzw+G_{0111}yzw \mod  \m^{[2]}.
\]
Suppose, for contradiction, that
\[
R:=W(k)[x,y,z,w]/(f+pG)
\]
is not $4$-quasi-$F$-split.
Since $R$ is not $2$-quasi-$F$-split, we obtain
\[
f\Delta_1(f+pG)\equiv a(a+1)G_{0111}^2x^3y^3z^3w^2 \in \m^{[4]}.
\]
Since $a(a+1)\neq 0$, we have $G_{0111}=0$.

Next, since $R$ is not $3$-quasi-$F$-split, we have $f\Delta_1(f+pG)^3 \in \m^{[8]}$.
The coefficient of $x^7y^5z^6w^6$ in $f\Delta_1(f+pG)^3$ is
\[
a(a+1)(G_{1011}^4+1+a^4)
\]
which must vanish.
Since $a(a+1)\neq 0$, we have $G_{1011}=a+1$.
%  The coefficient of $x^7y^3z^7w^6$ in $f\Delta_1(f+pG)^3$ is equal to
% \[
% G_{0111}^2(a^4+a^2G_{1011}^4 + a^3 +aG_{1011}^4)+G_{2101}^2(a+a^2) = a(1+a) G_{2101}^2,
% \]
% which must vanish.
% Since $a(a+1)\neq 0$, we have $G_{2101}=0$.
The coefficient of $x^5y^5z^7w^6$ in $f\Delta_1(f+pG)^3$ is equal to
\[
G_{0111}^4 (a^3 + a^2G_{1011}^2 + a^2 + aG_{1011}^2) + G_{1201}^2(a+a^2) = a(1+a) G_{1201}^2 =0,
\]
which must vanish.
Since $a(a+1)\neq 0$, we have $G_{1201}=0$.
% The coefficient of $x^3y^7z^7w^6$ in $f\Delta_1(f+pG)^3$ is equal to
% \[
% G_{0111}^6(a^2+a) + G_{0301}^2(a^2+a) = a(1+a) G_{0301}^2,
% \]
% which must vanish.
% Since $a(a+1)\neq 0$, we have $G_{0301}=0$.
% Then, combining with the above equalities, the coefficient of $x^6y^4z^7w^6$ in $f\Delta_1(f+pG)^3$ is equal to
% $G_{3001}$. Therefore, we have $G_{3001} =0$.

% Finally, since $R$ is not $4$-quasi-$F$-split, we have $f\Delta_1(f+pG)^7 \in \m^{[16]}$.
% The coefficient of $x^{15}y^{5}z^{14}w^{14}$ in $f\Delta_1(f+pG)^7$ is
% \[
% (a^5 + a^6 + (a+a^2+a^5+a^6)G_{1011}^4 + (a+a^2) G_{1011}^8)G_{1011}^4 + G_{3001}^4(a+a^2) = a^5(1+a)^5,
% \]
% which must vanish.
% Therefore, we obtain the contradiction.

Finally, since $R$ is not $4$-quasi-$F$-split, we have $f\Delta_1(f+pG)^7 \in \m^{[16]}$.
The coefficient of $x^{15}y^{13}z^{10}w^{14}$ in $f\Delta_1(f+pG)^7$ is
\begin{align*}
    &\mathrel{\phantom{=}} (a^2+a)G_{1201}^4G_{1011}^8+(a^2+a)G_{3001}^4G_{0111}^8+(a^6+a^5)G_{1201}^4 \\
    &+a(a+1)((a+1)^8+a^4)G_{1011}^4+a^5(a+1)^5G_{0111}^4+a(a+1)^{13} \\
    &= a^5(1+a)^5,
\end{align*}
which must vanish.
Therefore, we obtain the contradiction.

% \commentbox{The coefficient of $x^{15}y^{13}z^{10}w^{14}$ in $f\Delta_1(f+pG)^7$ is $(a^2+a)G_{1201}^4G_{1011}^8+(a^2+a)G_{3001}^4G_{0111}^8+(a^6+a^5)G_{1201}^4+a(a+1)((a+1)^8+a^4)G_{1011}^4+a^5(a+1)^5G_{0111}^4+a(a+1)^{13} = a^5(1+a)^5$}

(2) and (3) follow from the above proof.
One can also verify (4) by a direct computation.
Note that, within $f\Delta_1(f+pG)^3$, the exponents that must be examined are 
$(7,5,6,6), (5,5,7,6)$ (already computed above), together with 
\[
(7,7,5,6), (7,3,7,6), (6,6,6,6), (6,4,7,6), (5,7,6,6), (3,7,7,6),  (4,6,7,6). 
\]
\qedhere
% \commentbox{↓$G_{2101}$ detekuru}
% Finally, since $R$ is not $4$-quasi-$F$-split, we have $f\Delta_1(f+pG)^7 \in \m^{[16]}$.
% The coefficient of $x^{15}y^3z^{15}w^{14}$ in $f\Delta_1(f+pG)^7$ is
% \[
% f_{0030}^4f_{4011}^2f_{1310}(G_{3001}+abc)^2,
% \]
% which must vanish.
% Since $f_{0030},f_{4011},f_{1310}, G_{3001}+abc \neq 0$, we reach a contradiction.
\end{proof}

\begin{proposition}[$4A_2$]\label{4A_2}
Let notation be as in \cref{notation:fedder} assume
\[
(p,\deg(x),\deg(y),\deg(z),\deg(w))=(3,1,1,2,3)
\]
and
\[
f=w^2+z^3-x^2y^2(x+y)^2
\]
\begin{enumerate}
    \item The ring $W(k)[x,y,z,w]/(f+pG)$ is $3$-quasi-$F$-split for every homogeneous element $G$ of degree $6$.
    \item $\sht(W(k)[x,y,z,w]/f)=2$.
    \item $\sht(W(k)[x,y,z,w]/(f+p(xyz^2 + x^2z^2 + y^2z^2 -x^2y^2z)))=3$.
\end{enumerate}
\end{proposition}

\begin{proof}
First, we show the assertion (1).
Suppose, for contradiction, that
\[
R:=W(k)[x,y,z,w]/(f+pG)
\]
is not $3$-quasi-$F$-split.
Since $R$ is not $2$-quasi-$F$-split, we have
\[
f^2\Delta_1\bigl((f+pG)^2\bigr)
%\equiv -x^4y^4(x+y)^4z^6w^8-w^6f^2G^3 
\in \m^{[9]}.
\]
This is equivalent to
\begin{equation}
\label{eqn:4A2fDelta}
G_{2020}^3=G_{1120}^3=G_{0220}^3=1, \qquad G_{2210}^3= -1.
\end{equation}

Next, since $R$ is not $3$-quasi-$F$-split, we have 
\[
f^2\Delta_1\bigl((f+pG)^2\bigr)^4 \in \m^{[27]}.
\]
The coefficient of $x^{20} y^{10} z^{24} w^{26}$ in $f^2\Delta_1\bigl((f+pG)^2\bigr)^4$ is equal to
\[
G_{0220}^3G_{2020}^9 + G_{2020}^3 -1 =1,
\]
so we obtain the contradiction.
The assertions (2) and (3) follow from (\ref{eqn:4A2fDelta}).
% \commentbox{↓ dato $G_{0410}$ toka detekuru??}
% However, the coefficient of $x^{24}y^{24}z^{18}w^{24}$ in $f\Delta_1(f+pG)^3$ is
% \[
% 2\,G_{0220}^{3}+G_{1120}^{9}+2\,G_{2210}^{9}G_{1120}^{3}+G_{2210}^{12}+2\,G_{2020}^{3}=1,
% \]
% % \[
% % \,{\cred G_{0220}^{3}}+ {\cred 2} G_{1120}^{9}+2\,G_{2210}^{9}G_{1120}^{3}+G_{2210}^{12}+\,{\cred G_{2020}^{3}}=1,
% % \]
% %which is nonzero.
% This contradiction completes the proof.
\end{proof}

\begin{example}
The following two examples show that the properties of lifts of RDP del Pezzo surfaces observed in this section are specific to them.
\begin{enumerate}
\item
Let $k$ be an algebraically closed field of characteristic $2$.
Let $f \in W(k) [x,y,z,w, u,v]$ be
\[
f = x^5 + y^5 + z^5 + w^5 + u^5 +v^5.
\]
Then the ring $W (k)[x,y,z,w,u,v]/(f)$ is not quasi-F-split, since $\Delta_1 (f) \in \m^{[4]}$, and hence we have
\[
\theta (F_* a) := u(F_*(\Delta_1(f)a))  \in \m^{[2]}
\]
for any $a \in k[x,y,z,w,u,v]$ (we use notation in \cite{Yoshikawa25}*{Example 4.14}).
On the other hand, we have $\sht(W(k)[x,y,z,w,u,v]/(f + 2yzwuv)) =3$.
Indeed, we have $\sht(W(k)[x,y,z,w,u,v]/(f + 2yzwuv)) \geq 3$ since
\[
f \Delta_1 (f+ 2yzwuv) \in \m^{[4]},
\]
and we have
\[
f \Delta_1 (f+ 2 yzwuv)^{3} \equiv  x^5y^6z^6w^6u^6v^6 \notin \m^{[8]}.
\]
This example shows that quasi-$F$-splitting of the affine cone of a $W(k)$-lift of a smooth Fano variety of Picard rank 1 depends on the choice of the $W(k)$-lift in general.
%That is, the properties of the lifts of RDP del Pezzo surfaces observed in this section are specific to them.

\item
Let $k$ be an algebraically closed field of characteristic $3$.
Let $f \in W(k) [x,y,z,w, u]$ be
\[
f = x^4 + y^4 + z^4 + w^4 + u^4.
\]
Then, for any homogeneous element $G \in W(k)[x,y,z,w,u]$ of degree $4$, we have 
\[
\Delta_1((f+ pG)^{p-1}) = (p-1) f^{p(p-2)} \Delta_1 (f + pG) \in \m^{[9]}.
\]
Therefore, we have
\[
\theta (F_*a) = u (F_* (\Delta_1 ((f+pG)^{p-1}) a )) \in \m^{[3]}
\]
for any $a\in k[x,y,z,w,u].$
Therefore, rings $W(k)[x,y,z,w,u]/(f+pG)$ are not quasi-$F$-split, i.e., this gives an example of a smooth Fano threefold whose affine cone of any $W(k)$-lift is non-quasi-$F$-split.
\end{enumerate}
\end{example}

\subsection{Applications to weak del Pezzo surfaces}
Finally, we discuss applications of our results.
In particular, we deduce the quasi-$F$-splitting of smooth weak del Pezzo surfaces
and prove a Kodaira-type vanishing theorem for ample $\Q$-Cartier divisors on lifts of RDP del Pezzo surfaces.

\begin{theorem}
\label{weakdP}
Let $k$ be an algebraically closed field of characteristic $p>0$.
Let $X$ be a smooth projective scheme over $W(k)$ such that the relative dimension two and $-K_X$ is nef and big.
Then $X$ is quasi-$F$-split.
\end{theorem}

\begin{proof}
Since $-K_X$ is semiample, we obtain a projective birational morphism 
\(f \colon X \to Y\).
Then $Y$ is the lift of an RDP del Pezzo surface.
By \cref{RDP-dP}, $Y$ is quasi-$F$-split.

Consider the exact sequence
\[
0 \to F_*Q_{Y,p^lK_Y,n-1} \to Q_{Y,p^{l-1}K_Y,n} \to F_*\cO_{\var{Y}}(p^lK_{\var{Y}}) \to 0
\]
for every $n \geq 2$ and $l \geq 1$.
It follows that \(Q_{Y,K_Y,n} \simeq Rf_*Q_{X,K_X,n}\) for all $n \geq 1$.
Hence we obtain the commutative diagram
\[
\begin{tikzcd}
     \omega_{\var{Y}} \arrow[d,"\simeq"] \arrow[r] & Q_{Y,K_Y,n} \arrow[d,"\simeq"] \\
     Rf_*\omega_{\var{X}} \arrow[r] & Rf_*Q_{X,K_X,n}.
\end{tikzcd}
\]
Taking local cohomology and applying \cref{efs-local-coh}, we deduce that the quasi-$F$-splitting of $X$ follows from that of $Y$.
\end{proof}

\begin{theorem}[{Theorem \ref{intro:vanishing}}]
\label{vanishing}
Let $\var{X}$ be an RDP del Pezzo surface over an algebraically closed field $k$ of characteristic $p>0$.
Let $\var{A}$ be a nef and big Weil $\Q$-Cartier divisor on $\var{X}$.
Assume that there exists a lift $X$ of $\var{X}$ over $W(k)$ and a lift $A$ of $\var{A}$ on $X$ such that $A$ is a $\Q$-Cartier divisor.
Then $H^1(\cO_{\var{X}}(-\var{A}))=0$.
\end{theorem}

\begin{proof}
By \cite{Kawakami22}*{Lemma~2.13}, we may assume that $\var{X}$ is not log liftable.
Then the Picard rank of  $\var{X}$ is one by \cite{KN22}*{Theorem~1.3~(1),~Theorem~1.4~(0)--(2)}, and in particular, $\var{A}$ is ample. 
By \cite{KN22}*{Theorem~1.7~(3)}, we may assume that $p=2$ and $\var{X}$ is of type $(7A_1)$ or $(8A_1)$; in particular, $\var{X}$ is strongly $F$-regular.
By \cref{RDP-dP}, $X$ is quasi-$F$-split.
Thus it suffices to show that $\cO_X(-p^eA)$ is Cohen--Macaulay by \cref{thm:KV-vanishing}.
By the inversion of adjunction \cite{bcm-reg}*{Theorem~6.27} for perfectoid BCM-regularity together with the strong $F$-regularity of $\var{X}$, we obtain that $X$ is perfectoid BCM-regular.
Thus, the Cohen-Macaulayness follows from \cref{CM-divisorial} as below.
\end{proof}

\begin{lemma}\label{CM-divisorial}
Let $(R,\m)$ be a Noetherian local ring with $p \in \m$.
We assume $R$ is perfectoid BCM-regular.
Let $D$ be a $\Q$-Cartier Weil divisor on $\Spec{R}$, then $R(D)$ is Cohen-Macaulay.
\end{lemma}

\begin{proof}
Since $R$ is $\Q$-Cartier, there exists a finite extension $R \to S$ such that $R(D) \otimes_R S \simeq S$, and in particular, we have $R(D) \otimes_R R^+ \simeq R^+$.
Since $R$ is perfectoid BCM-regular, there exists a big Cohen-Macaulay $R$-algebra $B$ over $R^+$ such that $R \to B$ is pure.
Therefore, the homomorphism
\[
R(D) \to R(D) \otimes_R B \simeq B
\]
is pure.
Since $B$ is big Cohen-Macaulay, the $R$-module $R(D)$ is Cohen-Macaulay, as desired.
\end{proof}

\begin{corollary}[Theorem~\ref{intro:van-CM-ample}]\label{van-CM-ample}
Let $k$ be an algebraically closed field, and let $X$ be a lift of an RDP del Pezzo surface over $W(k)$.
Let $A$ be a nef and big Weil $\Q$-Cartier divisor on $X$ such that $\cO_X(A)$ is Cohen--Macaulay.    
Then we have $H^i(\cO_X(K_X+A))=0$ for $i>0$.
\end{corollary}

\begin{proof}
Let $\var{X}$ be the closed fiber of $X$.
We may assume $A$ and $\var{X}$ have no common component.
Since $\cO_X(K_X+A)$ is Cohen--Macaulay, we have an exact sequence
\[
0 \to \cO_X(K_X+A) \xrightarrow{\cdot p} \cO_X(K_X+A) \to \cO_{\var{X}}(K_{\var{X}}+\var{A}) \to 0,
\]
where $\var{A}:=A|_{\var{X}}$.
By \cref{vanishing}, we have $H^i(\cO_{\var{X}}(K_{\var{X}}+\var{A}))=0$ for $i>0$.
Hence, by the above exact sequence and Nakayama's lemma, we obtain $H^i(\cO_X(K_X+A))=0$ for $i>0$.
\end{proof}

\bibliographystyle{skalpha}
\bibliography{bibliography.bib}

\end{document}